\documentclass[a4paper]{article}
\usepackage{amsmath}
\usepackage{amssymb}
\usepackage{amsfonts}
\usepackage{amsthm}
\usepackage{tikz}
\usetikzlibrary{patterns, matrix}
\tikzstyle{NE-lines}=[pattern=north east lines, pattern color=black!45]
\newcommand{\fsum}[1]{\sum_{m\geq 0}m!\left(#1\right)^{m}}
\newcommand{\Sym}{\mathnormal{S}}
\newcommand{\Co}{\mathit{Co}}
\newcommand{\floor}[1]{\lfloor#1\rfloor}

\newcommand{\pattern}[4]{
 \raisebox{0.6ex}{
 \begin{tikzpicture}[scale=0.25, baseline=(current bounding box.center), #1]
   \foreach \x/\y in {#4}
     \fill[NE-lines] (\x,\y) rectangle +(1,1);
   \draw (0.01,0.01) grid (#2+0.99,#2+0.99);
   \foreach \x/\y in {#3}
     \filldraw (\x,\y) circle (5pt);
 \end{tikzpicture}}
}

\newtheorem{theorem}{Theorem}
\newtheorem{corollary}[theorem]{Corollary}

\title{Turning cycle restrictions into mesh patterns via Foata's fundamental transformation}
\author{Anders Claesson and Henning Ulfarsson}
\date{23 October 2023}

\begin{document}
\maketitle
\thispagestyle{empty}

\begin{abstract}
  An adjacent $q$-cycle is a natural generalization of an adjacent
  transposition. We show that the number of adjacent $q$-cycles in a
  permutation maps to the sum of occurrences of two mesh patterns under
  Foata's fundamental transformation. As a corollary we resolve
  Conjecture 3.14 in the paper ``From Hertzprung's problem to
  pattern-rewriting systems'' by the first author.
\end{abstract}

Let $q$ be a positive integer. Following Brualdi and Deutsch~\cite{Brualdi2012},
define an \emph{adjacent $q$-cycle} in a permutation as a cycle of the form
\[(i,i+1,\dotsc,i+q-1).
\]
In particular, an adjacent $1$-cycle is a fixed point and an adjacent
$2$-cycle is more commonly known as an \emph{adjacent transposition}. Brualdi and
Deutsch showed, among other things, that if $a_q(n,k)$ is the number of
permutations of $\{1,2,\dotsc,n\}$ that---when expressed as a product of disjoint cycles---have
exactly $k$ adjacent $q$-cycles, then
\[
  a_q(n,k) = \sum_{j=k}^{\floor{n/q}}(-1)^{k+j}\binom{j}{k}\frac{(n-(q-1)j)!}{j!}.
\]
Foata's fundamental transformation~\cite{foata} bijectively maps a
permutations $\pi$ with $k$ cycles to a permutation $\sigma$ with $k$
left-to-right minima by writing each cycle of $\pi$ so that its leftmost
element is the smallest, sorting the cycles in descending order with
respect to their first element, and reading the resulting permutation
$\sigma$ as a word from left to right. For instance,
\[
   \pi= 213967548
  = (5,6,7)(4,9,8)(3)(1,2)
   \;\;\mapsto\;\; \sigma=567498312.
\]
Consider the effect of this transformation on a fixed point (an adjacent
1-cycle) such as 3 in the permutation $\pi$ above. If it is not the
smallest element, then it is mapped to a left-to-right minimum in
$\sigma$ that is directly followed by another left-to-right minimum. In
terms of mesh patterns~\cite{Bra11} it is an occurrence of
\[
  s_1 = \pattern{}{2}{1/2,2/1}{0/0,0/1,1/0,1/1,1/2}
\]
Indeed, the shading to the southwest of the two points guarantees that
they are left-to-right minima, while the shading in the column between
the points guarantees that they are adjacent. If, on the other hand, the
fixed point is the smallest element, then it gets mapped to an occurrence
of $r_1=\pattern{scale=0.95}{1}{1/1}{0/0,1/0,1/1}$\vbox to 2.6ex{}.

Similarly, let us consider an adjacent transposition $(i, i+1)$ in
$\pi$. Depending on whether or not $(i, i+1)$ is the rightmost cycle of
$\pi$ (i.e., whether or not $i=1$) it is mapped to an occurrence of one
of the patterns
\[
  r_2 = \pattern{}{2}{1/1,2/2}{0/0,0/1,1/0,1/1,1/2,2/0,2/1,2/2}
  \quad\,\text{or}\quad
  s_2 = \pattern{}{3}{1/2,2/3,3/1}{0/0,0/1,0/1,0/2,1/0,1/1,1/2,1/3,2/0,2/1,2/2,2/3,3/2}
\]
in $\sigma$. For instance, $(1,2)$ in the example permutation $\pi$
corresponds to an occurrence of $r_2$ in $\sigma$.  Proceeding with
adjacent 3-cycles we have
\[
   r_3 = \pattern{}{3}{1/1,2/2,3/3}{
     0/0,0/1,0/2,1/0,1/1,1/2,1/3,2/0,2/1,2/2,2/3,3/0,3/1,3/2,3/3}
  \quad\,\text{and}\quad
   s_3 = \pattern{}{4}{1/2,2/3,3/4,4/1}{
     0/0,0/1,0/2,0/3,
     1/0,1/1,1/2,1/3,1/4,
     2/0,2/1,2/2,2/3,2/4,
     3/0,3/1,3/2,3/3,3/4,
             4/2,4/3}
\]
and the cycle $(5,6,7)$ in $\pi$ corresponds to a unique
occurrence of $s_3$ in $\sigma$.

It should now be clear how to define $r_q$ and $s_q$ for any $q\geq 1$,
and that the following theorem is true.


\begin{theorem}\label{main}
  Assume that $\pi\mapsto\sigma$ under Foata's fundamental
  transformation. For any $q \geq 1$, the number of adjacent $q$-cycles
  in $\pi$ is equal to the sum of the number of occurrences of $r_q$ and
  $s_q$ in $\sigma$.
\end{theorem}

The following corollary provides a generating function identity
conjectured by the first author~\cite[Conjecture 3.14]{Cl2022}.
\begin{corollary}
    We have
    \[\sum_{n\geq 0}|\Sym_n(p)|\mskip1mu x^n
        = \fsum{\frac{x}{1+x^2}},
        \;\text{ where }\,
        p = \pattern{}{3}{1/1,2/3/,3/2}{0/2,0/3,1/2,1/3,2/0,2/1,2/2,2/3,3/1,3/2}
    \]
    and $\Sym_n(p)$ denotes the set of permutations avoiding the mesh pattern $p$.
\end{corollary}
\begin{proof}
  Let us denote the right-hand side of the conjectured identity by
  $F(x)$, and let $A(x)$ be the generating function for the number of
  permutations of $[n]$ whose disjoint cycle decompositions have no
  adjacent transpositions. In other words, the coefficient of $x^n$ in
  $A(x)$ is $a_2(n,0)$; these numbers form sequence A177249 in the
  OEIS~\cite{OEIS}.  Brualdi and Deutsch~\cite{Brualdi2012} have shown
  that $A(x)$ satisfies the differential equation
  \[
    x^2(1+x^2)A'(x)-(1+x^2)(1-x-x^2)A(x)+1-x^2=0,\quad A(0)=1.
  \]
  Term-wise differentiation yields
  \[
    \frac{d}{dx}\biggl(\frac{F(x)}{1+x^2}\biggr)
    = \frac{(1-x-x^2)F(x)-1+x^2}{x^2(1+x^2)}
  \]
  and using this it is easy to verify that $F(x)/(1+x^2)$ satisfies the same
  differential equation as $A(x)$. It thus suffices to show that
  $\sum_{n\geq 0}|\Sym_n(p)|\mskip1mu x^n = (1+x^2)A(x)$, or, equivalently,
  \begin{equation}\label{the-eq}
    |\Sym_n(p)| \,=\, a_2(n,0) + a_2(n-2,0)\quad\text{for $n\geq 2$.}
  \end{equation}
  As a special case of Theorem~\ref{main} we find that Foata's
  fundamental transformation provides a one-to-one correspondence
  between permutations without adjacent transpositions and permutations
  that avoid $r_2$ and $s_2$.  By symmetry (reverse followed by inverse)
  we may equivalently consider permutations avoiding the two patterns
  \[
    r_2' = \pattern{}{2}{1/2,2/1}{0/0,0/1,0/2,1/0,1/1,1/2,2/0,2/1} \quad\text{and}\quad
    s_2' = \pattern{}{3}{1/1,2/3/,3/2}{0/1,0/2,0/3,1/1,1/2,1/3,2/0,2/1,2/2,2/3,3/1,3/2}
  \]
  By the Shading lemma~\cite{shading}, the pattern $p$ is
  \emph{coincident} with the pattern $s_2'$, in the sense that
  $\Sym_n(p) = \Sym_n(s_2')$ for all $n\geq 0$. By conditioning on
  whether a permutation avoiding $s_2'$ also avoids $r_2'$ or contains
  $r_2'$ we find that
  \[
    \Sym_n(p) =  \Sym_n(r_2', s_2') \cup \bigl(\Co_n(r_2') \cap \Sym_n(s_2')\bigr),
  \]
  where the union is disjoint and $\Co_n(r_2')$ denotes the set of
  permutations that contain $r_2'$. Note that a permutation $\pi$
  contains $r_2'$ and avoids $s_2'$ precisely when it is the direct sum
  $\pi=21\oplus\sigma$ of the permutation $21$ and a permutation
  $\sigma$ that avoids $r_2'$ and $s_2'$. This establishes
  equation~\eqref{the-eq} and concludes the proof.
\end{proof}

As above, assume that $\pi\mapsto\sigma$ under Foata's fundamental
transformation. A direct consequence of the definition of this map
is that the number of cycles of $\pi$ equals the number of left-to-right
minima of $\sigma$, which in turn is the number of occurrences of the
mesh pattern
\[
  \pattern{}{1}{1/1}{0/0}
\]
in $\sigma$. We have shown that the statistic ``number of
fixed points'' translates to the mesh pattern statistic
\[
  \pattern{scale=0.95}{1}{1/1}{0/0,1/0,1/1} \;+\, 
  \pattern{}{2}{1/2,2/1}{0/0,0/1,1/0,1/1,1/2}
\]
and more generally that the number of adjacent $q$-cycles translates to
$r_q+s_q$. A fixed point $\pi(i)=i$ of $\pi$ is said to be \emph{strong}
if $j<i\Rightarrow \pi(j)<\pi(i)$ and $j>i\Rightarrow \pi(j)>\pi(i)$;
see \cite[Ex.\ 1.32b]{EC1}. In term of mesh patterns, strong fixed
points are occurrences of $\pattern{scale=0.8}{1}{1/1}{1/0,0/1}$ in
$\pi$, and it is easy to see that they map to occurrences of
$\pattern{scale=0.8}{1}{1/1}{0/0,1/1}$ in $\sigma$, sometimes called
\emph{skew strong fixed points}~\cite{Bra11}.  What other properties of
the cycle structure of $\pi$ can be neatly expressed in terms of
occurrences of mesh patterns in $\sigma$?  Are the examples presented
here special cases of a more general phenomena?

\subsection*{Acknowledgements}
This work was started at Schloss Dagstuhl (Leibniz-Zentrum für
Informatik), seminar 23121, and we thank the institute and the
organizers for giving us the opportunity to participate.

\bibliographystyle{plain}
\bibliography{hertz}

\end{document}